\documentclass[fontsize=11pt,a4paper,DIV=12]{scrartcl}

\usepackage[T1]{fontenc}
\usepackage{amsmath, amsfonts, amssymb, amsthm}
\usepackage{graphicx}
\usepackage{xcolor}
\usepackage{multirow, array}
\usepackage{hyperref, hypcap}
\usepackage{thmtools, cleveref}
\usepackage{subcaption}
\usepackage{mathtools}

\newcommand{\arr}{\mathcal{A}}

\newtheorem{cor}{Corollary}

\definecolor{darkblue}{rgb}{0,0,0.7} % darkblue color
\newcommand{\defn}[1]{\textsl{#1}} % emphasis of a definition

\newtheorem{conj}{Conjecture}
\newtheorem{definition}{Definition}
\newtheorem{lemma}{Lemma}
\newtheorem{theorem}{Theorem}
\newtheorem{remark}{Remark}

\begin{document}

\title{On Triangles in Colored Pseudoline Arrangements}
\author{
    Yan Alves Radtke\,\thanks{Institut für Mathematik, Technische Universität Berlin, Germany. alves@math.tu-berlin.de} \and Bal\'azs Keszegh\,\thanks{HUN-REN Alfréd Rényi Institute of Mathematics and ELTE Eötvös Loránd University, Budapest, Hungary. keszegh@renyi.hu.\\ Research supported by the ERC Advanced Grant ``ERMiD'' and by the EXCELLENCE-24 project no.~151504 Combinatorics and Geometry of the NRDI Fund. This research has been implemented with the support provided by the Ministry of Innovation and Technology of Hungary from the National Research, Development and Innovation Fund, financed under the  ELTE TKP 2021-NKTA-62 funding scheme.} \and Robert Lauff\, \thanks{Institut für Mathematik, Technische Universität Berlin, Germany. lauff@math.tu-berlin.de}
}

\date{}

\maketitle

\begin{abstract}
    We consider the faces in pseudoline arrangements in which the pseudolines are colored with two colors. Björner, Las Vergnas, Sturmfels,  White, and Ziegler conjecture the existence of a two-colored triangle in such arrangements. We consider variants of this problem. We show that in any non-trivial two-coloring of a pseudoline arrangement there exists a two-colored triangle or quadrangle. 
    We also investigate the existence of a bichromatic triangle assuming certain structures on the coloring.
   
    Previously, several authors investigated the chromatic number and independence number of hypergraphs whose vertices correspond to the pseudolines of an arrangement and the hyperedges correspond to the faces of the arrangement. We show that the maximum of the independence numbers of such hypergraphs is $\lceil \frac{2}{3}n-1\rceil$. We also prove that if we only consider the triangular faces then this maximum becomes $n-\Theta(\log n)$.
\end{abstract}

\section{Introduction}

An \defn{Euclidean pseudoline arrangement} is a finite collection of bi-infinite, simple curves called \defn{pseudolines} in the Euclidean plane, such that they pairwise cross in exactly one point, which we will call \defn{crossing}. 

An arrangement is \defn{simple}, if no three pseudolines intersect in a common point. We only consider simple arrangements.

A pseudoline arrangement gives rise to a collection of \defn{vertices}, \defn{edges} and \defn{faces}, see Figure~\ref{fig:plaflip}.
We call a bounded face a \defn{triangle} resp. \defn{quadrangle}, if it is supported by exactly three resp. four pseudolines. 
Sometimes we consider the area between three pseudolines and call it a \defn{non-empty triangle}.

The minimum  and maximum  number of triangles in pseudoline arrangements are known\cite{minnumbertriangles,maxnumbertriangles}.
Other questions about triangles remain open. 
We will consider arrangements whose pseudolines are colored blue and red so that at least one pseudoline of each color exists. We call such arrangements \defn{bicolored}. In 1993, Björner, Las Vergnas, Sturmfels, White, and Ziegler asked about the existence of bichromatic triangles in bicolored arrangements.
\begin{conj}[{\cite[p. 280]{orientedmatroids}}]\label{conj:bichromatic triangle}
	Every bicolored arrangement has a bichromatic triangle.
\end{conj}

Another way to look at this problem is to consider the \defn{triangle-pseudoline incidence graph}.
As its vertices, take the triangles and pseudolines and add an edge between a pseudoline and a triangle if the pseudoline supports the triangle. 
It is intuitive to ask about the connectivity of this graph.

\begin{conj}[{\cite[p. 278]{orientedmatroids}}]\label{conj:connected_incidence graph}
    The triangle-pseudoline incidence graph of a pseudoline arrangement is connected. \footnote{The original (weaker) conjecture concerned projective arrangements.}
\end{conj}

Note that Conjecture~\ref{conj:bichromatic triangle} and Conjecture~\ref{conj:connected_incidence graph} are equivalent. 
The existence of bichromatic triangles in a bicolored straight line arrangement has a simple proof: Shift the blue sub-arrangement up and down. 
Consider the first moment the combinatorics of the arrangement changes. 
At this point, a blue-blue crossing moved over a red line or a red-red crossing moved over a blue line. 
The crossing and the line form a bichromatic triangle in the original arrangement.

A generalization of this result was given in \cite{approachingarrangements}:
An arrangement $\arr=\{\ell_1,\dots,\ell_n\}$ is \defn{approaching}, if every pseudoline $\ell_i$ is the graph of a function $f_i$, such that $f_i-f_j$ is strictly monotone increasing, for $i<j$. In \cite{approachingarrangements} they show that as these pseudolines can be shifted such that the collection of curves remains an arrangement of pseudolines, Conjecture~\ref{conj:bichromatic triangle} holds for all approaching arrangements.

The number of approaching arrangements has asymptotics similar to the class of all arrangements, but examples of small arrangements that are not isomorphic to an approaching arrangement are known.

Although the proof in the case of approaching arrangements is very elegant and intuitive, not much progress has been made towards the general case so far. 
We present results about easier variants of the question.

Also towards the goal of understanding triangles in pseudoline arrangements, we consider the \defn{(pseudo)line-triangle hypergraph}, which is the 3-uniform hypergraph whose vertices correspond to the (pseudo)lines and a triple forms a hyperedge if the corresponding (pseudo)lines form a triangle. 
Then Conjecture \ref{conj:bichromatic triangle} claims that every true 2-coloring of this hypergraph has a non-monochromatic hyperedge.
We present a theorem about the independence number of this hypergraph, i.e., we consider how large the discrepancy of the cardinalities of the two color classes in an arrangement can be while not including any monochromatic triangles.

\subsection{Preliminaries}
We will consider Euclidean arrangements as \defn{marked arrangements}, which have a fixed unbounded face, which we will call the \defn{north face}.
This gives us a canonical numbering of the pseudolines, after choosing the north face, see Figure~\ref{fig:plaflip}.
This induces
an orientation on every triples of lines: we say that a triple of pseudolines has "-" orientation, if the middle pseudoline goes above the crossing of the other two, otherwise it is $+$, see Figure~\ref{fig:plaflip}.

\begin{definition}[\cite{sweeps}]
The function $\sigma: \binom{[n]}{k}\rightarrow \{-,+\}$ is a rank $k$ signotope on $n$ elements, if for all $$S\coloneqq\{x_1<x_2<\dots<x_{k+1}\}\subseteq[n],$$
    the sequence $$\sigma(S\setminus x_1),\sigma(S\setminus x_2),\dots, \sigma(S\setminus x_{k+1})$$ has at most one sign change.
\end{definition}
\begin{figure}[ht]
    \centering
    \includegraphics[width=0.75\linewidth]{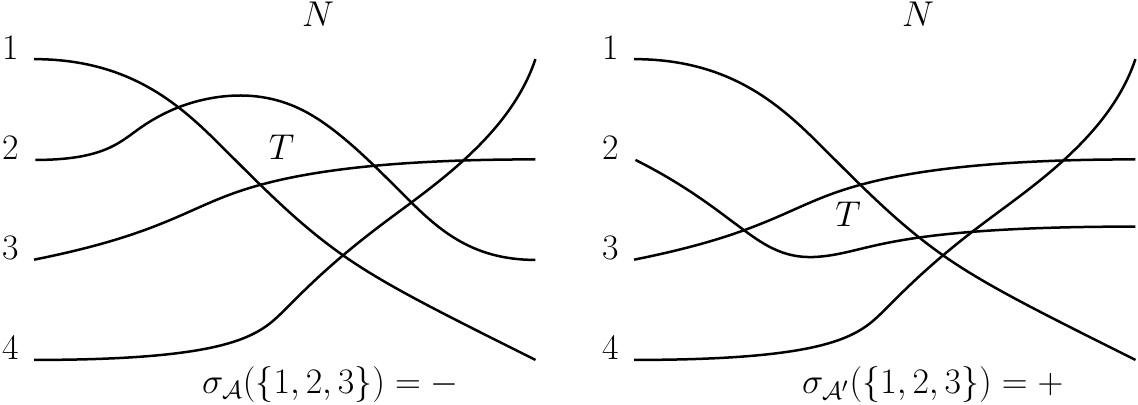}
    \caption{ The north face is designated by $N$. The face $T$ is a triangle. The crossings (1,2) and (1,3) are vertices, which are connected by an edge. The left arrangement corresponds to the all "$-$" signotope. The right arrangement is obtained by flipping $T$. }
    \label{fig:plaflip}
\end{figure}
The orientation of triples of pseudolines of $\arr$ give rise to the rank 3 signotope $\sigma_{\arr}$. Conversely, every rank 3 signotope can be seen as triangle orientations of a marked pseudoline arrangement \cite{sweeps}. It is possible to locally mutate a pseudoline arrangement $\arr$ by \defn{flipping} a triangle $T$ to obtain $\arr'$, i.e. changing the orientation of the pseudolines supporting $T$. 
This is equivalent to $\sigma_\arr$ and $\sigma_{\arr'}$ differing by a single value, corresponding to $T$. See Figure~\ref{fig:plaflip}. 

The most elementary tool to locate triangles is the so called \defn{sweeping lemma for pseudolines}.
\begin{lemma}[\cite{sweeps}]\label{lemma:sweepingpla}
    Let $\ell$ be a pseudoline in a marked arrangement. If there is a crossing above $\ell$, then there is a crossing above $\ell$ that forms a triangle supported by $\ell$. The same holds for crossings below $\ell$. 
\end{lemma}

There are two a-priori different orders on the set of signotopes with the same parameters: the \defn{inclusion order}, where $\sigma_1 \leq \sigma_2$, when $\sigma_1^{-1}\{+\} \subseteq \sigma_2^{-1}\{+\}$ and the \defn{single-step inclusion order}, where  $\sigma_1 \leq \sigma_2$, if there is a sequence of $-$ to $+$ flips that transform $\sigma_1$ into $\sigma_2$.
The following was proven by Felsner and Weil.
\begin{theorem}[\cite{theoremonhbo}]

\label{thm:ssioandio}
    The single-step inclusion order and the inclusion order on rank three signotopes coincide.
\end{theorem}

Two interesting arrangements are the so called \emph{cyclic arrangements}.
They correspond to the all $+$ or all $-$ signotopes. 
See Figure~\ref{fig:cyclic}.
\begin{figure}[ht]
    \centering
    \includegraphics[width=0.75\linewidth]{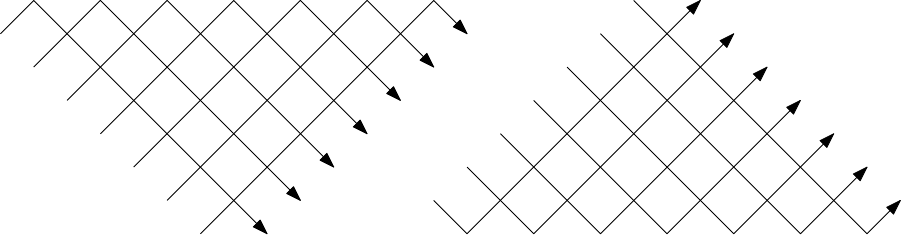}
    \caption{ The two cyclic arrangements on 8 pseudolines. Note that these arrangements are realizable by lines.}
    \label{fig:cyclic}
\end{figure}

\subsection{Our results about bichromatic faces}
We consider multiple weaker variants of Conjecture~\ref{conj:bichromatic triangle}. 
We first assume certain structures on the coloring of the arrangements.
For completeness, we state the following, which is already implicit in \cite{radtkeetal2023}.
\begin{restatable}{theorem}{fewreds}
\label{thm:fewreds}
	A bicolored arrangement with at most 5 red pseudolines has a bichromatic triangle.
\end{restatable}

This implies that if we color an arrangement with $n/5$ colors and use all of them, then there is a non-monochromatic triangle.
It also implies the conjecture holds for all arrangements with $n\leq 11$.
We call an (unmarked) bicolored pseudoline arrangement \defn{block-bicolored}, if we can choose a north cell, such that the first pseudolines in the numbering induced by the north cell are red and the rest are blue, see Figure~\ref{fig:example_block_bicolored}.
\begin{figure}
    \centering
    \includegraphics[width=0.75\linewidth]{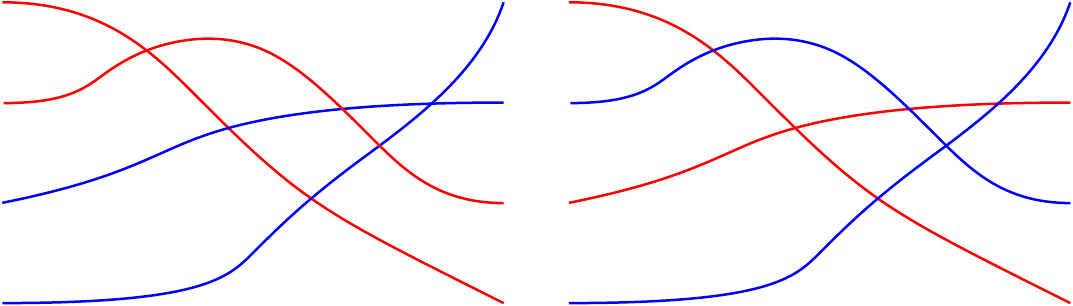}
    \caption{ The arrangement on the left is block-bicolored. The arrangement on the right is not.}
    \label{fig:example_block_bicolored}
\end{figure}
The following result was first proven by Felsner [personal communication]. Here we provide a simpler proof. 

\begin{restatable}{theorem}{blockbicolored}
\label{thm:blockbicolored}
    If a pseudoline arrangement is \defn{block-bicolored} it has a bichromatic triangle.
\end{restatable}

We then broadened the question and considered the existence of bichromatic faces with low complexity.
 \begin{restatable}{theorem}{triorquad}\label{thm:always tri or quadrangle}
	Every bicolored arrangement contains a bichromatic triangle or a bichromatic quadrangle.
\end{restatable} 
\begin{remark}
    In the case where we find a quadrangle, this quadrangle is red-red-blue-blue.
\end{remark}

\subsection{(Pseudo)line-face and (pseudo)line-triangle hypergraphs}

In this section we always assume that the arrangement of (pseudo)lines is simple and that the (pseudo)lines pairwise cross.

Let $H_{line-face}(\arr)$ denote the line-face hypergraph of the arrangement of a set of lines $\arr$, i.e., the hypergraph whose vertices correspond to the lines and a subset forms a hyperedge if the corresponding lines form a face in the arrangement (including the unbounded faces). $H_{psline-face}(\arr)$ is defined the same way for arrangements of a set of pseudolines. Note that every line arrangement is a pseudoline arrangement.  We consider the non-bounded faces as well. 

Properties of these hypergraphs were regarded earlier by Bose et al.\cite{first} and then their initial results were improved by Ackerman et al. \cite{second} and Balogh and Solymosi \cite{BS18}. %The line-triangle hypergraph $H_\Delta(L)$ is the subhypergraph of the line-face hypergraph with hyperedges corresponding to triangles.

Let $\alpha$ be the independence number of a hypergraph, that is, the largest size of an independent set. The following theorem summarizes previous knowledge about the relevant parameters of line-face hypergraphs:
\begin{theorem}[\cite{first, second,BS18}]\label{thm:oldLF}
    \begin{itemize}
        \item[]
        \item $\Omega( \sqrt{n\log n}) = \min\limits_{|\arr|=n \text{}}\alpha(H_{psline-face}(\arr))\le \min\limits_{|L|=n}\alpha(H_{line-face}(\arr))\le n^{5/6+o(1)}$,       
        \item $n^{1/6-o(1)} \le \max\limits_{|\arr|=n}\chi(H_{line-face}(\arr)) \le \max\limits_{|\arr|=n}\chi(H_{psline-face}(\arr)) = O( \sqrt{n/\log n})$,
         \item $n/2\le \max\limits_{|\arr|=n}\alpha(H_{line-face}(\arr))\le \max\limits_{|\arr|=n}\alpha(H_{psline-face}(\arr))< \frac{2}{3}n$.
    \end{itemize}
\end{theorem}

Let us summarize the history of the above theorem.
First in \cite{first} the lower bound of $\sqrt{n+1}-1$ was proved for $\min_{|\arr|=n}\alpha(H_{line-face}(\arr)$ and then this was improved to $\Omega( \sqrt{n\log n})$ in \cite{second}. Their proof works also for pseudolines and even for non-simple arrangements.
The upper bound was shown in \cite{BS18} by giving a construction that improves the best bound for the so called Erd\H os $(3,4)$-problem and then dualizing it (the dual relation between the two problems was mentioned already in \cite{second}).
These directly imply the bounds on $\max_{|\arr|=n}\chi(H_{line-face}(\arr))$ and $\max_{|\arr|=n}\chi(H_{psline-face}(\arr))$ although there was previously a lower bound of $\Omega(\log n/\log\log n)$ in \cite{first}. We remark that so far these results hold also if we consider only those hyperedges that correspond to  bounded triangular faces of the arrangement.

The upper bound on $\max_{|\arr|=n}\alpha(H_{line-face}(\arr))$ is again from \cite{first}, and their proof works for pseudolines as well. The lower bound was also proved in \cite{first} by showing that for the line arrangement in which the lines are in convex position, its line-face hypergraph is two-colorable, therefore one of the color classes must be an independent set of size at least $n/2$. Notice that this also implies that $\min_{|\arr|=n}\chi(H_{face}(\arr))=2$. They also showed that in any line arrangement whose line-face hypergraph is two-colorable, both color classes must be of size $n/2+O(\sqrt{n})$, so using only such line arrangements we won't be able to improve the lower bound on  $\max_{|\arr|=n}\chi(H_{line-face}(\arr))$ considerably. 

Our first result about such hypergraphs essentially solves the problem of maximizing $\alpha$:

\begin{theorem}\label{thm:psline-face}
 \begin{itemize}
        \item[]
        \item
    $\max\limits_{|\arr|=n}\alpha(H_{psline-face} (\arr))=\lceil \frac{2}{3}n-1\rceil$,
    \item $\lfloor \frac{2}{3}(n-1)\rfloor\le \max\limits_{|\arr|=n}\alpha(H_{line-face} (\arr))\le \lceil \frac{2}{3}n-1\rceil$.
 \end{itemize}
\end{theorem}

Note that for $n=3k$ (resp. $n=3k+1$)  the lower and upper bounds for $\max\limits_{|\arr|=n}\alpha(H_{line-face} (\arr))$ are both equal to $2k-1$ (resp. $2k$). However, for $n=3k-1$ there remains a slight gap as these bounds give only that $2k-2\le \max\limits_{|\arr|=n}\alpha(H_{line-face} (\arr))\le 2k-1$. For small values $k=1,2,3,4$ the upper bound gives the truth, as we will see. Note that for pseudolines the upper bound gives always the truth, $\max\limits_{|\arr|=n}\alpha(H_{psline-face} (\arr))=2k-1$ for $n=3k-1$.

\smallskip
Instead of the (pseudo)line-face hypergraph, we can consider the (pseudo)line-triangle hypergraph, which we denote by $H_{line-triangle}(\arr)$ (resp. by $H_{psline-triangle}(\arr)$). We consider the non-bounded triangular faces as well. The next is a direct corollary of Theorem \ref{thm:oldLF}, using that the (pseudo)line-triangle hypergraph is a subhypergraph of the (pseudo)line-face hypergraph. We also use that in the construction in \cite{BS18} it is enough to consider only bounded triangular faces to obtain a small independence number.

%already considering only the triangular faces the independence number is small (and therefore the chromatic number is large).

\begin{cor}
    \begin{itemize}
        \item[]
        \item $\Omega( \sqrt{n\log n})= \min\limits_{|\arr|=n \text{}}\alpha(H_{psline-triangle}(\arr))\le \min\limits_{|\arr|=n} \alpha(H_{line-triangle}(\arr))\le n^{5/6+o(1)}$,
        \item $n^{1/6-o(1)} \le \max\limits_{|\arr|=n}\chi(H_{line-triangle}(\arr)) \le \max\limits_{|\arr|=n}\chi(H_{psline-triangle}(\arr)) = O( \sqrt{n/\log n})$.
    \end{itemize}
\end{cor}

Thus our knowledge about the line-face and line-triangle hypergraph is the same for the above parameters. Our second result about such hypergraphs shows, however, that the behaviour of $\max\alpha$ changes considerably:

\begin{restatable}{theorem}{maxalpha}\label{thm:line-triangle} 

$\max\limits_{|\arr|=n}\alpha(H_{psline-triangle} (\arr))=n-\Theta(\log n)$ and \\$\max\limits_{|\arr|=n}\alpha(H_{line-triangle} (\arr))=n-\Theta(\log n)$.
\end{restatable}

\section{Proofs}

\begin{figure}[h]
  \centering
  \includegraphics[width=0.75\textwidth]{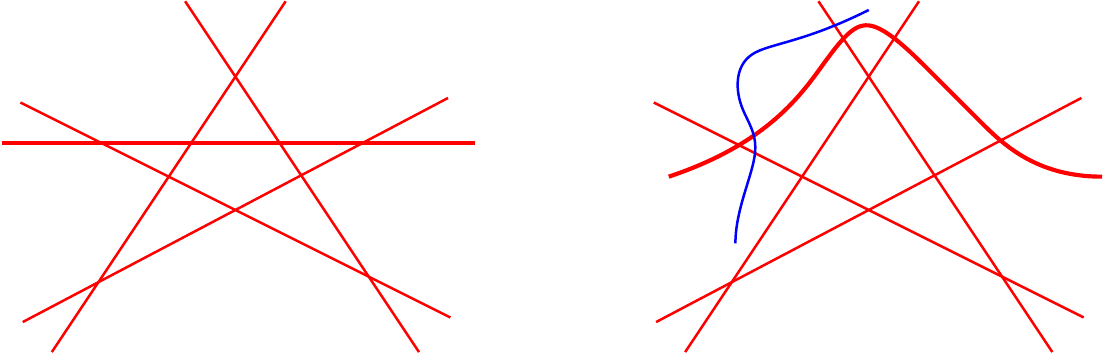}
  \caption{The 5-star is on the left side. On the right side, there is the situation after flipping the first triangle and finding a new triangle.}
  \label{fig:5star}
\end{figure}

\begin{proof}[Proof of \Cref{thm:fewreds}]
Let $\arr$ be a bicolored pseudoline arrangement with at most 5 red pseudolines and at least a single blue pseudoline. We start with a (finite) induction on the number of red lines. If there are at most two red pseudolines, both lines support some triangles, $\arr$ by Lemma~\ref{lemma:sweepingpla}. Since there are less then 3 red pseudolines in $\arr$ these triangles have to be bichromatic, giving the induction start.

Assume now that the red sub-arrangement has an extremal pseudoline $\ell$, that is, a pseudoline so that every crossing is on one of its sides. If $\ell$ is also extremal in $\arr$, we delete it. By induction, the remaining arrangement contains a bichromatic triangle which cannot be intersected by $\ell$, as $\ell$ was assumed to be extremal. If $\ell$ is not extremal in $\arr$ then it has a side which contains crossings, but no crossings of two red pseudolines. We apply Lemma~\ref{lemma:sweepingpla} to find a triangle supported by $\ell$ on this side. The triangle is bichromatic, since it can not consist of only red-red crossings. Hence we can assume that the red sub-arrangement has no extremal lines. As any arrangement on 4 pseudolines contains an extremal one, we can now assume that the number of red lines is exactly 5. It can easily be checked that the only arrangement of 5 pseudolines containing no extremal one is the 5-star, seen in \Cref{fig:5star}.

For the 5-star we can argue that the five triangles have to be empty: Assume one of them is not empty and consider the red line separating it from the red 5-gon. This red line must support a triangle on both of its sides as it is not extremal in $\arr$. This triangle must then be bichromatic as the only monochromatic triangle on one of its sides is not empty by assumption.
From now on we assume that the five triangles of the red sub-arrangements are empty. This implies that the pentagonal face of the 5-star is also empty.
We also assume that there is more than one blue pseudoline in $\arr$, otherwise Lemma~\ref{lemma:sweepingpla} gives us the existence of a bichromatic triangle.

Since there are at least two blue pseudolines, there is a blue-blue crossing.
Since no blue pseudoline intersects the triangles of the 5-star, the blue-blue crossing lies in an unbounded face of the 5-star. Then there is a red pseudoline that separates the blue crossing from the pentagon. 
We apply Lemma~\ref{lemma:sweepingpla} to this red pseudoline, for its side where the blue-blue crossing is.

There we find a triangle. If this triangle is not a triangle of the 5-star, we are done.
Otherwise, flip this triangle and apply Lemma~\ref{lemma:sweepingpla} to the same pseudoline and to the same side again.
The second triangle found has to be bichromatic. If this is a triangle also in the original arrangement, we are done.
Assume now that the second triangle was not a triangle in the original arrangement, i.e., the first flip created the second triangle. Note that the newly created triangle has to be supported by two red pseudolines that were involved in the first flip and thus the third pseudoline is blue, which cannot go through the pentagon by our earlier assumption. Consequently, the case where the first flip creates a new triangle is essentially unique and is shown in Figure~\ref{fig:5star}. Then the five triangles of the 5-star were not all empty, contradicting our earlier assumption.
\end{proof}

\begin{figure}[htbp]
  \centering
  \includegraphics[width=0.75\textwidth]{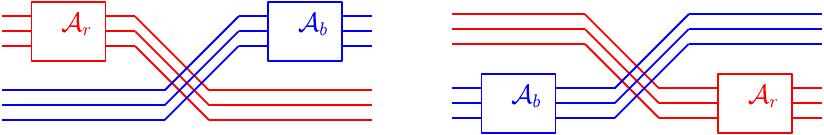}
  \caption{Both extremal arrangements with prescribed sub-arrangements}
  \label{fig:maxandmiarr}
\end{figure}

\begin{proof}[Proof of \Cref{thm:blockbicolored}]
        Given an unmarked, block-bicolored, arrangement $\arr$, we call the arrangement induced by the red pseudolines $\arr_r$ and the arrangement induced by the blue pseudolines $\arr_b$.
        We want to consider $\arr$ as a marked arrangement.
        To this end, we choose the north pole given by the definition of a block-bicolored arrangement.
        
    Consider the arrangements shown in Figure~\ref{fig:maxandmiarr}. 
    Call the arrangement to the left $\arr_{min}$, the one to the right $\arr_{max}$.
    
    In $\arr_{min}$, every triple of lines, not all of them of the same color, the pseudoline with the highest index is going below the crossing of the other two. This implies that every bichromatic triple has $-$ orientation. 
    
   Analogous, every bichromatic triple $\arr_{max}$ has orientation $+$. The remaining, monochromatic, triples in $\arr, \arr_{max}$ and $\arr_{min}$ have the same orientation. By the observation above clearly $\arr_{min} \leq \arr \leq \arr_{max}$ holds in the inclusion order. By Theorem~\ref{thm:ssioandio}, there is a sequence (possibly of length zero!) of "$-$ to $+$" triangle flips, transforming $\arr$ into $\arr_{max}$. 
    Since both arrangements agree on the monochromatic triples, the sequence only consists of bichromatic triangle flips. 
A similar statement holds for $\arr_{min}$.

Clearly $\arr_{min} \neq \arr_{max}$, so not both $\arr=\arr_{min}$ and $\arr=\arr_{max}$ can hold, so not both flip sequences can be empty. The first triangle flipped in such a sequence is a triangle of $\arr$. By the observation above, this implies at least one bichromatic triangle in $\arr$. Moreover if both $\arr \neq \arr_{min}$ and $\arr \neq \arr_{max}$, $\arr$ has two bichromatic triangles, of opposite orientation.
\end{proof}

Note that this proof gives us a sequence of arrangements which look like we shift a sub-arrangement to the left, similar to the proof for line arrangements.

In the following we will use a sweeping lemma for lenses, which was stated in \cite{radtkeetal2023}. 
\begin{lemma}[Lens Sweeping Lemma]\label{lem_parr sweep}
    Let $Q$ be a lens bounded by two curves $L$ and $R$. Assume there is a collection of curves inside $Q$ which pairwise intersect at most once and where every curve intersects $L$ and $R$ exactly once. 
    If there is a crossing inside $Q$, there is a triangle that is supported by $L$.
\end{lemma}

\begin{proof}[Proof of Theorem~\ref{thm:always tri or quadrangle}]
\begin{figure}[ht]
    \centering
    \includegraphics[width=0.75\linewidth]{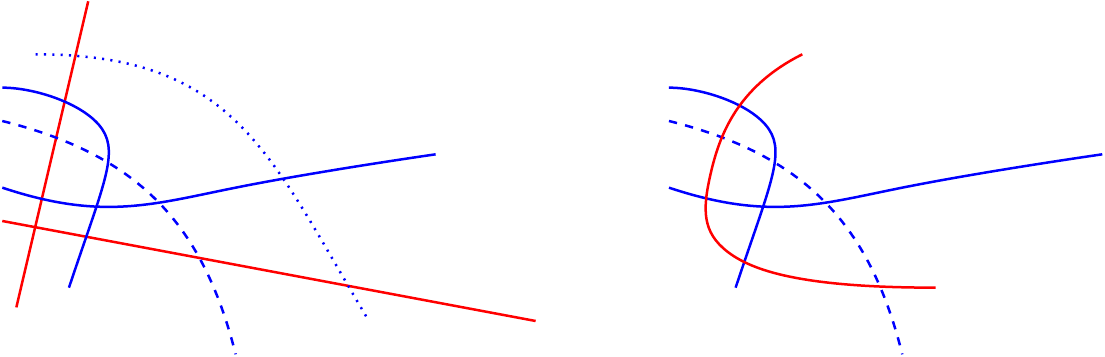}
    \caption{Illustration for the proof of \Cref{thm:always tri or quadrangle}. 
    %The dotted line is the one we first chose in the proof. 
    The dashed blue and both red lines form a lens.} \label{fig:alwaystriorquad}
\end{figure}

	In a bicolored arrangement, we consider the sub-arrangement induced by the red pseudolines and a single blue pseudoline. 
    There is a blue-red-red triangle $T$ in this subarrangement incident to the blue pseudoline, by the sweeping lemma. 
    We reintroduce the remaining blue pseudolines and consider the blue pseudolines that cross $T$. 
    If there is one that intersects both red sides  then we have a smaller blue-red-red triangle, induction.
    Otherwise all blue pseudolines intersect the blue side and one of the red sides. 
    Then taking just the part of the arrangement inside the triangle and lengthening all lines to infinity in both directions, plus joining the two red halflines that form the red-red sides of $T$ to form a single curve, we get a lens with only transversal arcs as in Lemma~\ref{lem_parr sweep}. 
    If there are no arcs going through the lens, $\ell_a,\ell_b$ and $\ell_c$ are a bichromatic triangle.
    If there is no crossing inside the lens, then we find a bichromatic triangle inside the lens, which consists of the vertex $(\ell_a,\ell_b)$ and the unique arc that sees this vertex.
    
    Otherwise we can sweep the red line of the lens towards the blue by triangle flips and hence we find a triangle formed by it and two blue pseudolines crossing the lens. 
    Then we have a red-blue-blue triangle on that side, before merging the two red lines this was either a red-blue-blue triangle or a red-red-blue-blue quadrangle, finishing the proof.
    
\end{proof}

\begin{figure}[ht]
    \centering
    \includegraphics[width=1\linewidth]{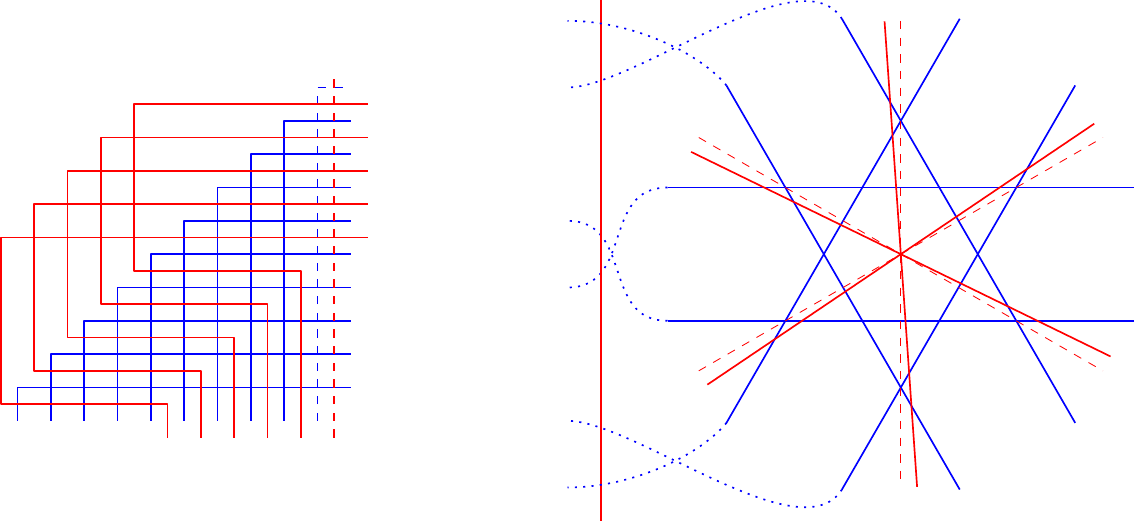}
    \caption{Left: $10$ (resp. $9$, without the dashed lines) blue pseudolines form an independent set in the pseudoline-face hypergraph of $16$ (resp. $14$) lines. Right: $6$ lines form an independent set in the line-face hypergraph of $10$ lines.}
    \label{fig:nhalf}
\end{figure}

\begin{figure}[ht]
    \centering
    \includegraphics[width=1\linewidth]{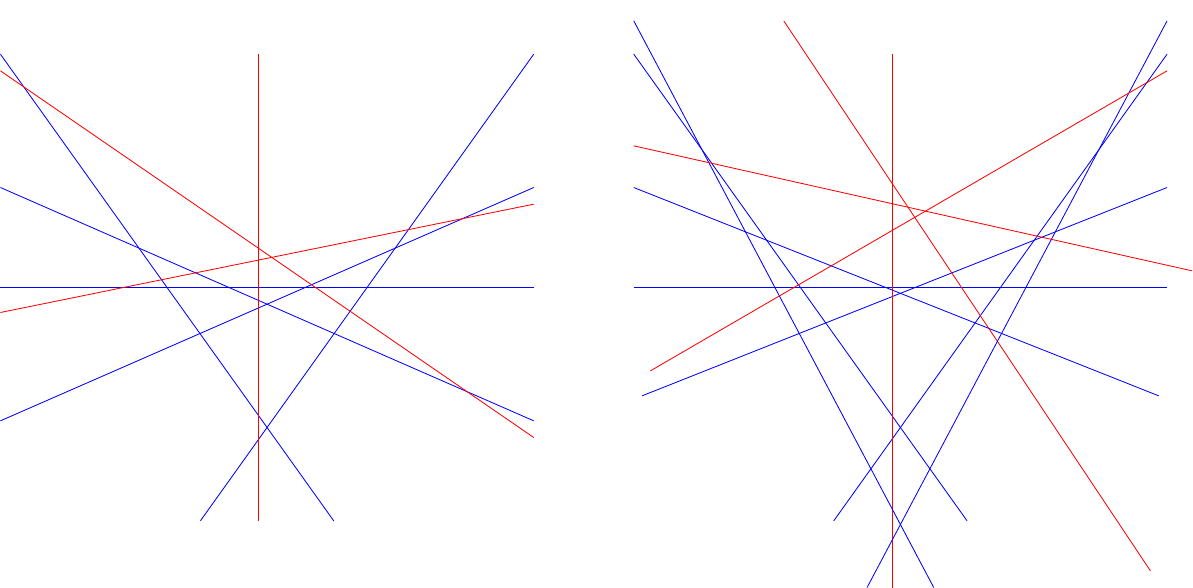}
    \caption{Left: $5$ blue lines form an independent set in the line-face hypergraph of $8$ lines. Right: $7$ blue lines form an independent set in the line-face hypergraph of $11$ lines.}
    \label{fig:nhalf5-7}
\end{figure}

\begin{figure}[ht]
    \centering
    \includegraphics[width=1\linewidth]{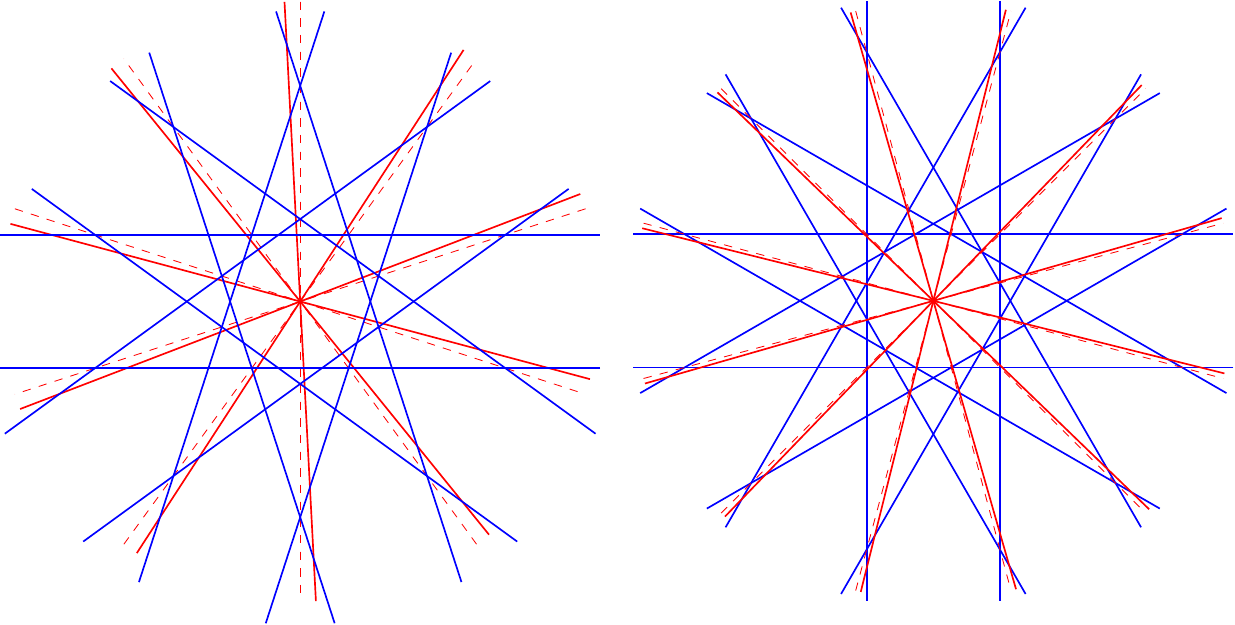}
    \caption{    Left: $10$ blue lines form an independent set in the line-face hypergraph of 15 lines, having parallel pairs of blue lines. Right: $12$ lines form an independent set in the line-face hypergraph of 18 lines, having parallel pairs of blue lines.}
    \label{fig:nhalf10-12}
\end{figure}

\begin{proof}[Proof of Theorem \ref{thm:psline-face}]
The upper bound for pseudolines follows from Theorem \ref{thm:oldLF}. For the sake of completeness and as we will afterwards strengthen the argument slightly, first we refresh their proof. Given a pseudoline arrangement $L$ of size $n$, let $s$ be the size of an independent set $I$ in the pseudoline-face hypergraph. Then the pseudolines of $I$ form $s(s+1)/2+1$ cells. Each pseudoline in $J=L\setminus I$ intersects $s+1$ such cells. Note that $|J|=n-s$. As $I$ is an independent set, each cell must be intersected at least once. Therefore, as $s$ is an integer, $$(n-s)(s+1)\ge s(s+1)/2+1,$$
$$n\ge \frac{3}{2}s+1/(s+1),$$
$$s<\frac{2}{3}n,$$
$$s\le \lceil \frac{2}{3}n-1\rceil.$$

\begin{lemma}\label{lem:oneface}
    If $n=3k-1$ and in a set of $n$ pseudolines $L$ a subset of $2k-1$ (blue) pseudolines forms an independent set $I$, then there must be a cell $C$ in the subarrangement of the blue pseudolines $I$, such that $C$ is unbounded, every (red) pseudoline in $J=L\setminus I$ goes through $C$ and for every other cell of $J$ there is exactly one pseudoline of $J$ that goes through it.
\end{lemma}

\begin{proof}
For examples for $k=5$ see the left side of Figure \ref{fig:nhalf} without the dashed lines and for $k=3,4$ see Figure \ref{fig:nhalf5-7}. The cases $k=1,2$ can also be easily drawn with straight lines.

First we slightly improve the above computation for $n=3k-1$. Assume we have a set of lines $L$ of size $n=3k-1$ and an independent subset of (blue) lines $I$ of size $s=2k-1$. Take the (red) lines in $J=L\setminus I$ in any order and note that each line in $J$ except for the first one intersects the previous lines in some cells (at least one!) of the arrangement of $I$ and so in these cells it actually does not cross a new cell. Therefore instead of the above inequality we can write the following stronger one: $$(s+1)+(n-s-1)s=(n-s)s+1\ge s(s+1)/2+1.$$
Now using that $n=3k-1$ and $s=2k-1$ we get $$k(2k-1)+1\ge (2k-1)(2k)/2+1.$$
As the above is actually an equality, it implies that there can be no further loss, in particular the last line intersects all previous ones in the same cell of $I$. This is true for any order of the lines, in particular any line can be last, which implies that every pair of lines in $J$ must actually intersect in the same cell $C$ of the arrangement of $I$. $C$ must be an unbounded cell, as otherwise it would have at most $2k-1$ sides, and each of the $k$ lines in $J$ would need to intersect two different ones (otherwise they meet in another cell, which would be a further loss), which is impossible. Also as there can be no further loss, every other cell is met by exactly one line of $J$.
\end{proof}

Now we give lower bound constructions. We start with the case of pseudolines. That is, for every $n$ we need a set of pseudolines $L$ in which there is a subfamily $L'$ of lines of size $\lceil \frac{2}{3}n-1\rceil$ that form an independent set in the pseudoline-face hypergraph. See Figure \ref{fig:nhalf} (left) for an illustration. Let $k\ge 1$ integer. First we take $2k$ (blue) pseudolines forming a staircase (their arrangement is actually the same as for lines in convex position). Then we take $k+1$ (red) lines as in Figure \ref{fig:nhalf} (left). Observe that every face of the arrangement of blue pseudolines is crossed by some red pseudoline. Thus in the arrangement of blue and red pseudolines there is no face which is bounded only by blue pseudolines.
This implies that in this construction of $3k+1$ pseudolines the $2k$ blue pseudolines form an independent set in the pseudoline-face hypergraph. Notice that $2k=\lceil \frac{2}{3}(3k+1)-1\rceil$, as required. 

By deleting one blue line (say, the dashed one on the Figure) we get a construction with $3k$ pseudolines and $2k-1$ blue pseudolines, and notice that $2k-1=\lceil \frac{2}{3}(3k)-1\rceil$.

Finally by deleting the rightmost blue line and the rightmost red line (in the order as their bottom vertical parts cross the $x$-axis, the dashed lines on the figure) we get a construction with $3k-1$ pseudolines and $2k-1$ blue pseudolines, and notice that $2k-1=\lceil \frac{2}{3}(3k-1)-1\rceil$. Notice also that as guaranteed by Lemma \ref{lem:oneface} every red-red intersection happens in the same face of the blue arrangement.
Altogether, we got constructions for every value of $n$. 

Now we give constructions with straight lines. Consider the regular $2k$-gon. After extending all its sides to lines, we get a line arrangement of $2k$ lines. 
Note that the pairs of lines arising from opposite sides of the $2k$-gon intersect at infinity. Except for the central face of size $2k$, every other face has at most $4$ lines on its boundary.

We now add $k$ red lines, which go through opposite vertices of the star formed by the $2k$ blue lines (for $k$ even, these are the $k$ diagonals of the $2k$-gon; for $k$ odd, these lines cut into half the edges of the $2k$-gon, see Figure \ref{fig:nhalf10-12}). 
Every blue quadrangle is formed by two pairs of lines, where every pair is consecutive along the $2k$-gon.
By the symmetries of the construction, this implies that every quadrangle formed by blue lines has exactly two vertices, which lie on a red line. The blue triangles also have either two vertices met by a red line or their interior is met by  a red line. Also the two-faces are all at a vertex of the $2k$-star and so have a red line intersecting them in their inside.
Thus, after rotating the red lines slightly around their intersection point, all faces of the blue line arrangement are hit by a red line in their interior. We can also slightly translate the red lines so that there are no $3$ lines intersecting in the same point. However, we have parallel pairs of blue lines that do not intersect yet.
Thus, from every pair of parallel blue lines, we rotate one slightly, such that the resulting intersection has small (i.e., large negative) $x$-coordinate. Now add a vertical red line, such that all blue-blue crossings are to the right of it. See \ref{fig:nhalf} (right). Thus this new red line crosses every face introduced when parallel blue lines were made to cross. This construction has $3k+1$ lines out of which $2k$ form an independent set. Deleting one or two blue lines gives a construction for $3k$ or $3k-1$, respectively, thus we got constructions for every value of $n$. 

As noted this gives only $2k-2$ independent lines when $n=3k-1$, one less than the upper bound. For small values $k=1,2,3,4$ we can get matching constructions, see Figure \ref{fig:nhalf5-7} for $k=3,4$.
\end{proof}

\begin{proof}[Proof of Theorem \ref{thm:line-triangle}]
We begin with a lower bound by constructing a bicolored pseudoline arrangement $\arr$ with $n$ blue and $2^n$ red lines not containing any monochromatic red triangles. Then the set of red lines is independent in $H_\Delta(\arr)$. More precisely we show the following statement by induction. There are two pseudoline arrangements, one partial and one proper with $n$ blue and $2^n$ red lines not containing a monochromatic triangle such that the set of red and blue lines cross transversely, i.e. along a simple closed curve which encloses all crossings we see the color of the intersecting line change only 4 times and the blue cross in one of the two arrangements (making the arrangement proper) while there are no blue crossings in the other (making the arrangement partial). This is clearly the case for 1 blue and 1 red line giving the induction start. In this case the two arrangements are equal. Assume the statements holds for some $n$. 

    Place the construction with blue crossings on top and the one without blue crossings on the bottom as a block each and match the blue lines. Cross the red lines of the two blocks as can be seen in \Cref{fig:nomonchromtriangles}. The only monochromatic triangles formed are between the two blocks and the new red crossings and they are all red. Add a new blue line crossing (and hence destroying) all these triangles and then make it cross all the blue lines under everything. It is clear that this does not introduce any blue triangles as the bottom block does not contain blue crossings. The construction for the case that there are no blue crossings works in the same way, only that we use itself as a block twice and make no new blue crossings below. This finishes the construction.

\begin{figure}[t]
    \centering
    \includegraphics[width=0.5\linewidth]{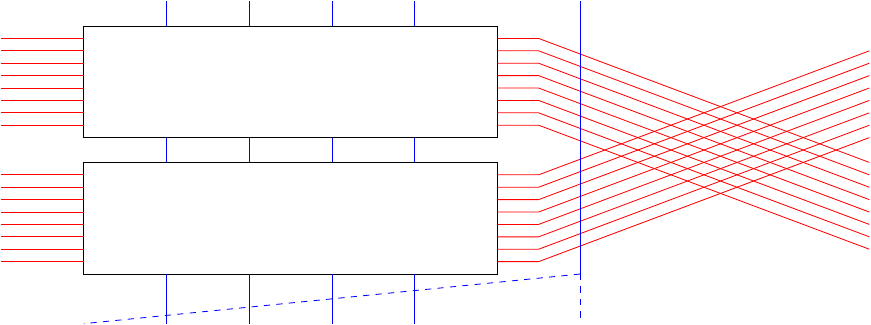}
    \caption{
    The bottom block has no blue crossings and the top block either has all blue crossings or no blue crossings. The extra blue line either crosses all other blue lines below, or it crosses none.}
    \label{fig:nomonchromtriangles}
\end{figure}

\begin{figure}
\begin{subfigure}{.3\textwidth}
  \centering
  \includegraphics[width=.8\linewidth]{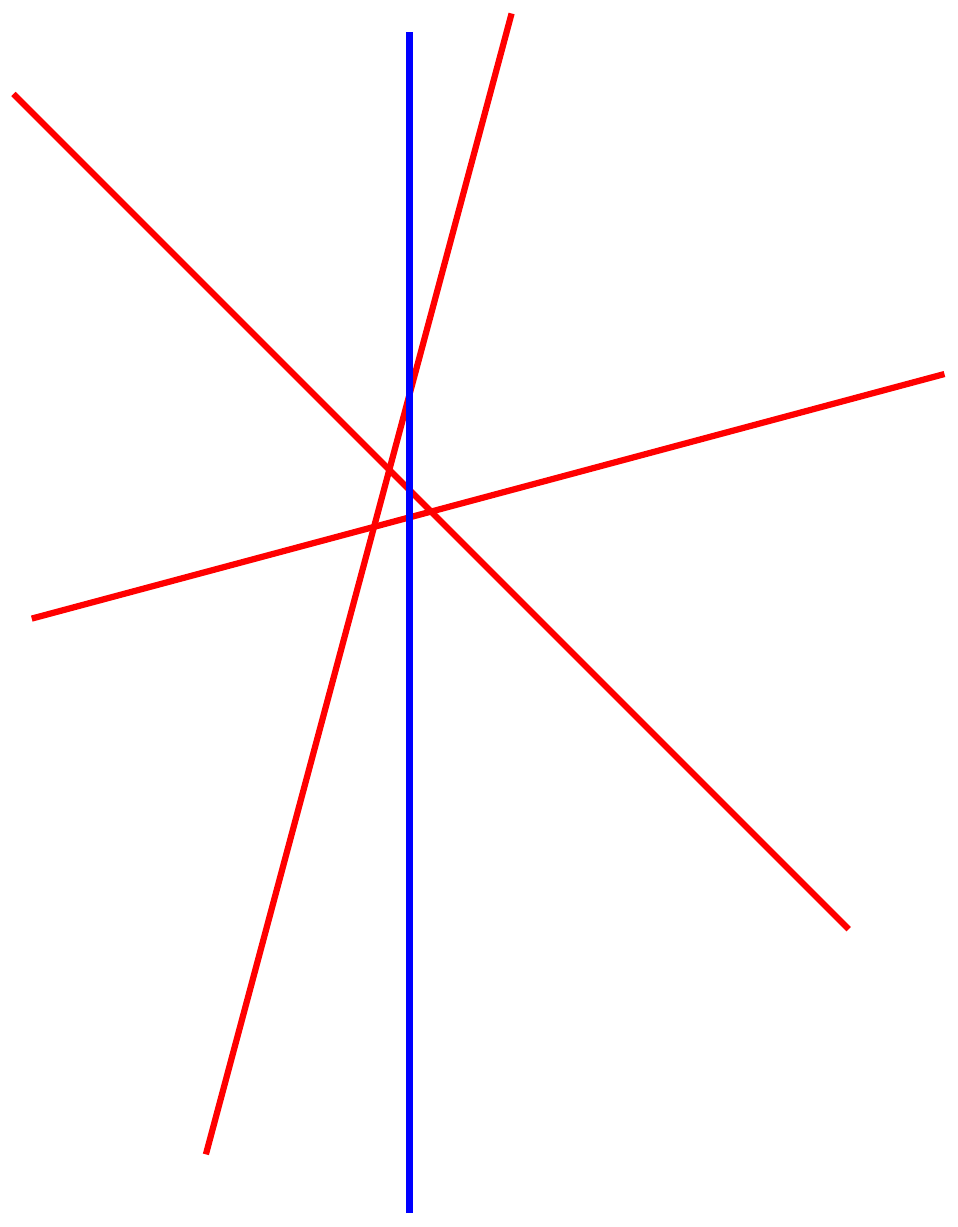}
  \label{fig:sfig1}
\end{subfigure}%
\begin{subfigure}{.3\textwidth}
  \centering
  \includegraphics[width=.8\linewidth]{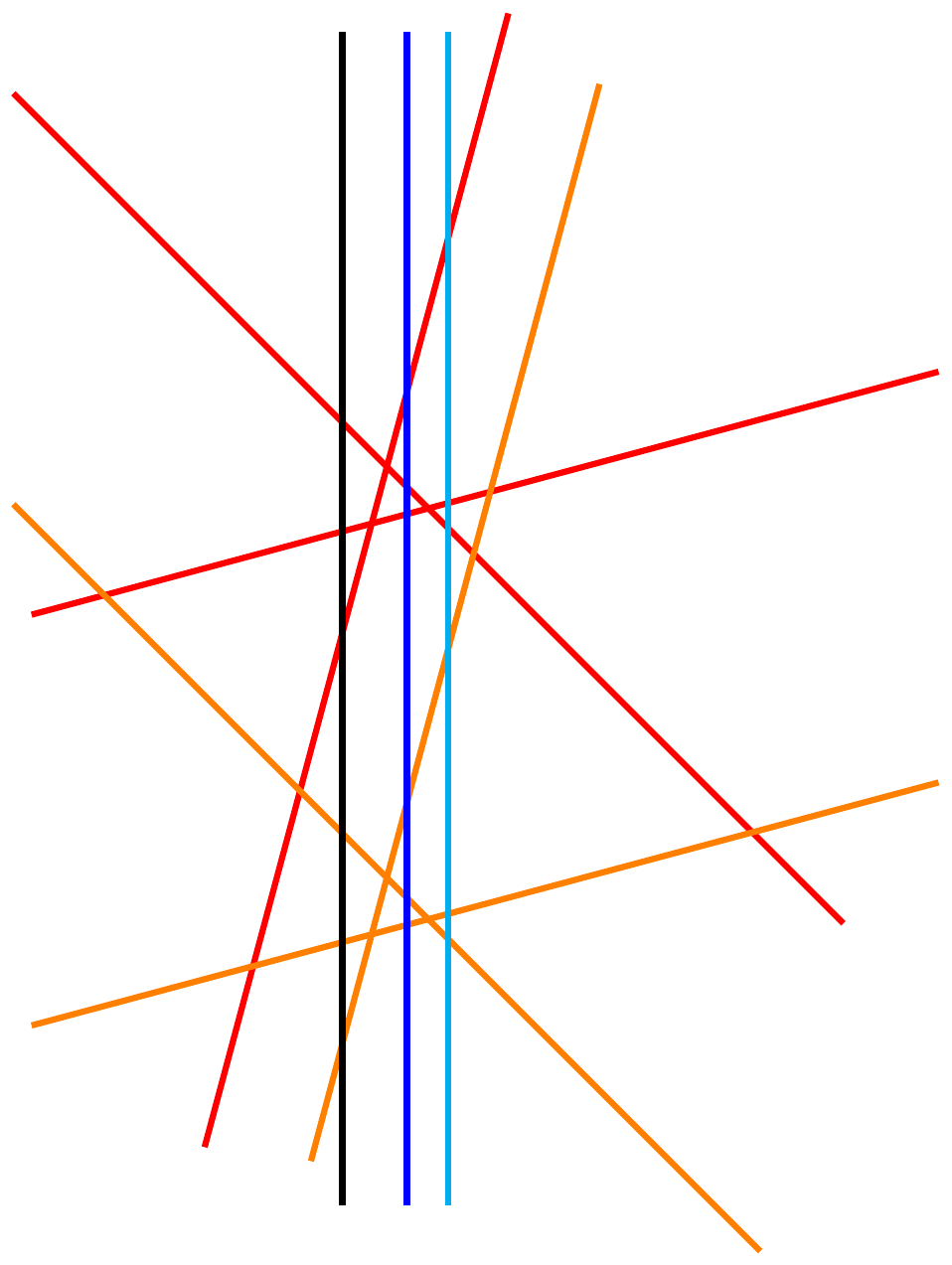}
  \label{fig:sfig2}
\end{subfigure}
\begin{subfigure}{.3\textwidth}
  \centering
  \includegraphics[height=5cm,width=.8\linewidth]{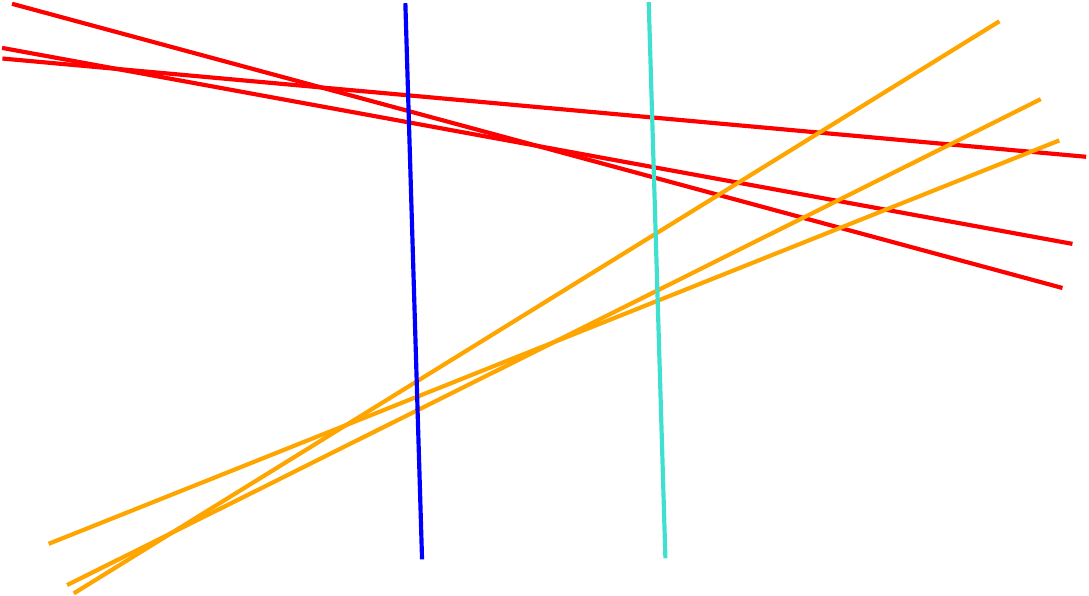}
  \label{fig:sfig3}
\end{subfigure}
\caption{The first arrangement is the induction basis.
In the second arrangement, the orange lines are a translated copy of the red lines and the cyan and black lines are parallel to the blue line. We obtain the arrangement to the right by moving the black line to infinity.}
\label{fig:real_non_mon_chrom}
\end{figure}

    Now we aim to make a similar arrangement with straight lines that avoids monochromatic red triangles.
    We argue that the construction without blue-blue crossings is realizable by a line arrangement. 
    For this, assume by induction, that there is a realization by lines, where, without loss of generality, all blue lines are vertical. 
    Take a copy of the red line arrangement and shift it vertically downwards, along the blue lines.
    We will consider the x coordinates of the red-red crossings. 
    Let $A$ be the set of crossings by a pair of red lines which do not belong to the same copy of the red subarrangement. 
    Conversely, let $B$ be the set of the red-red crossings, where both lines belong to the same copy of the red subarrangement. 
    Note that the x-coordinates of the points in $B$ are independent of the shift of the copy. 
    By shifting far enough, we can assume that for no $(a_1,a_2),(b_1,b_2) \in B$ and $(c_1,c_2)\in A$, it holds that $a_1<c_1<b_1$.
    We insert two new vertical lines, such that $B$ is between them and no point of $A$ is between them. 
    Of these two lines, we color the right one blue. 
    We send the left one by a projective transformation to infinity. Note that no red lines are parallel anymore.
    This line arrangement is now isomorphic to the next inductive step of the construction without blue-blue crossings. For a sketch, see Figure~\ref{fig:real_non_mon_chrom}.
    To obtain a line arrangement without parallel lines, we slightly perturb the slope of the blue lines.
    This construction implies that Theorem \ref{thm:line-triangle} holds for line-triangle hypergraphs, since we constructed a bicolored line arrangement without a red monochromatic triangle, with $2^n$ red and $n$ blue lines.

    We proceed with the upper bound by showing that every pseudoline arrangement with  $2^n$ red and less than $\frac{cn-2}{2}$ blue lines contains a monochromatic red triangle. Let $\arr$ be an arrangement of $2^n$ red pseudolines. Then, by Erd\H os-Szekeres for the point-line dual $\arr$ contains a subset of $cn$ red lines forming the cyclic arrangement. As every triangle in a pseudoline arrangement contains an empty triangle, the $cn-2$ triangles of this cyclic arrangement contain empty red triangles of $\arr$. Note that any pseudoline added to a cyclic arrangement intersects at most two of its empty triangles. Therefore, any set of $\frac{cn-2}{2}-1$ blue pseudolines added to $\arr$ leave an arrangement where one of the empty red triangles found before remains empty and monochromatic, finishing the proof.
\end{proof}

\section{Discussion}

The conjecture about the existence of a bichromatic triangle is solved if there are at most $5$ red pseudolines. One particular case of $6$ red pseudolines  which we could not solve in general (that is, for any addition of blue pseudolines) and find especially interesting is depicted in Figure \ref{fig:6lines}. Note that in any extension with blue pseudolines, as every red pseudoline must have an incident triangle on both sides by Lemma \ref{lemma:sweepingpla}, either one of them is bichromatic and we are done, or no gray triangles on the figure can be crossed by a blue pseudoline.

\begin{figure}[hbt!]
  \centering
  \includegraphics[width=0.4\textwidth]{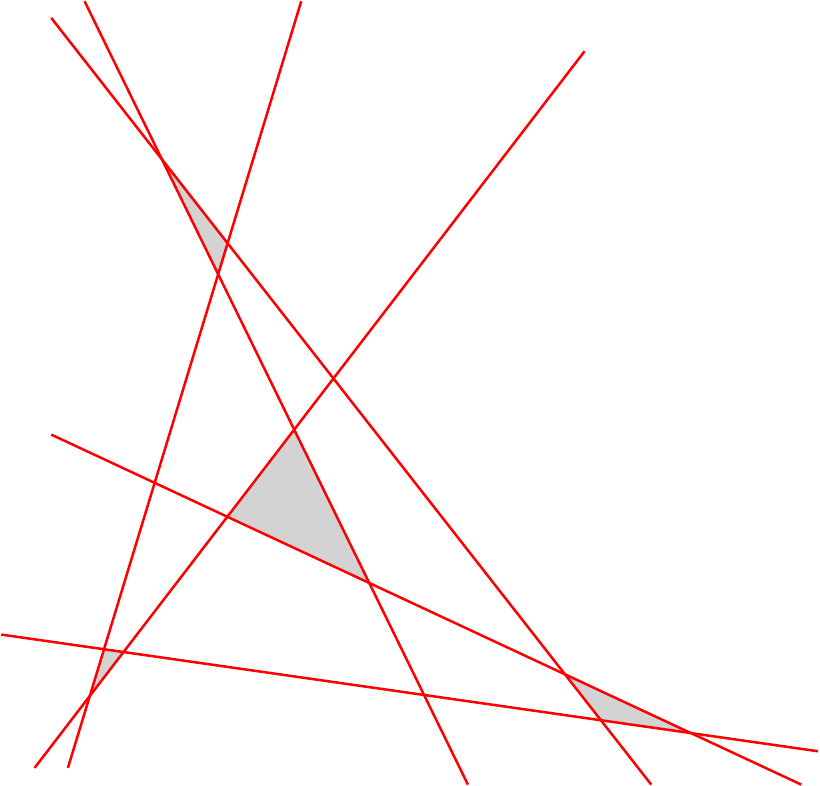}
  \caption{Six red pseudolines to which we need to add several blue pseudolines.}
  \label{fig:6lines}
\end{figure}

The bounds regarding (pseudo)line-face hypergraphs and (pseudo)line-triangle hypergraphs  still have large gaps except for $\max\alpha$, which is handled by Theorem \ref{thm:psline-face} and Theorem \ref{thm:line-triangle}.
Concerning the slight gap in Theorem \ref{thm:psline-face}, it is likely that  for straight lines we have the same answer as for pseudolines, that is, $\max\limits_{|\arr|=n}\alpha(H_{line-face} (\arr))=\lceil \frac{2}{3}n-1\rceil$? We suspect that with more effort our constructions for $n=8,11$ (see Figure \ref{fig:nhalf5-7}) could be generalized to any $n=3k-1$ yet we leave this gap of $1$ in our result as an open problem.

\begin{figure}[hbt!]
  \centering
  \includegraphics[width=0.2\textwidth]{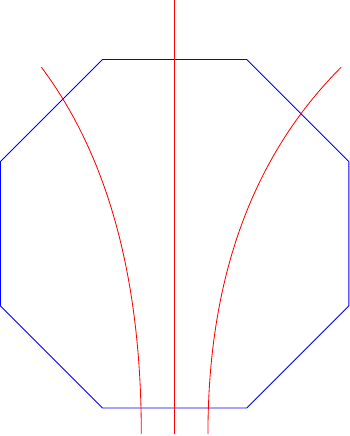}
  \caption{$3$ pseudolines cut an $8$-gon into pentagons.}
  \label{fig:kgon}
\end{figure}

Towards proving Theorem \ref{thm:always tri or quadrangle} we tried to but could not show the following statement which we think is interesting in its own too and conjecture to be true.
\begin{conj}
   Given a blue $k$-gon $P$ with at least $k-4$ red pseudolines passing through, there exists a bichromatic triangle or quadrangle on its inside boundary.
\end{conj}
It is easy to see that $k-4$ is tight, as one can take $k-5$ red pseudolines having no intersection inside the $k$-gon and partitioning the $k$-gon into pentagons, thus there are no triangles nor quadrangles at all, let alone bichromatic ones, see Figure \ref{fig:kgon}.

\vspace{1cm} 

\noindent\textbf{Acknowledgement}

We thank the organizers of the Order and Geometry Workshop 2024 and Stefan Felsner for helpful discussions. We thank Gábor Tardos for valuable discussions that eventually lead to Theorem \ref{thm:psline-face}.

%\newpage
\bibliography{sources.bib}
\end{document}